\documentclass{siamltex}
\usepackage{amssymb,amsmath,amsbsy,mathrsfs,eucal,amsfonts,eucal,amscd,exscale}
\usepackage{graphics}
\usepackage{verbatim}
\usepackage{appendix}
\usepackage{rotating}
\usepackage{enumerate}
\usepackage{xspace}
\usepackage{comment}
\usepackage{cleveref}
\usepackage{bm}
\usepackage{mathtools}
\DeclarePairedDelimiter{\ceil}{\lceil}{\rceil}

\usepackage{color}
\definecolor{red}{rgb}{1,0,0}

\definecolor{blue}{rgb}{0,0,1}

\usepackage{tikz}
\usetikzlibrary{arrows}
\usetikzlibrary{decorations.pathreplacing} 

\numberwithin{equation}{section}
\numberwithin{table}{section}
\numberwithin{figure}{section}

\newcommand{\bld}[1]{\boldsymbol{#1}}

\newcommand{\dm}{\Omega}

\newcommand{\bdry}{\partial \Omega}

\newcommand{\Nh}{{\mathcal{N}_{h}}}
\newcommand{\Nhi}{{\mathcal{N}_{h}^I}}
\newcommand{\Nhb}{{\mathcal{N}_h^B}}

\newcommand{\gradv}{\bld{\nabla}}

\newcommand{\norm}[1]{\|\,#1\,\|}

\newcommand{\MA}{Monge-Amp\`ere }
\newcommand{\dist}{\textrm{dist}}
\newcommand{\wha}{w^{\epsilon}}
\newcommand{\whb}{w_{\epsilon}}
\newcommand{\Cea}{\mathcal{C}^{\epsilon}}
\newcommand{\Ceb}{\mathcal{C}_{\epsilon}}
\newtheorem{prop}{Proposition}[section]
\newtheorem{remark}{Remark}[section]

\newcommand{\p}{\partial}
\newcommand{\bbZ}{\mathbb{Z}}
\newcommand{\bbR}{\mathbb{R}}
\newcommand{\calC}{\mathcal{C}}

\newcommand{\eps}{\epsilon}
\newcommand{\nab}{\nabla}
\newcommand{\mct}{\mathcal{T}_h}

\title{Rates of convergence in $W^2_{\lowercase{p}}$-norm for the Monge-Amp\`ere Equation}

\author{Michael Neilan
\thanks{Department of Mathematics, University of Pittsburgh.  The first author was partially supported by
NSF grants DMS-1541585 and DMS-1719829.   }
\and 
Wujun Zhang
\thanks{Department of Mathematics, Rutgers University. The second author was supported by the start up funding at Rutgers Unviersity.}}

\begin{document}
\maketitle

\begin{abstract}
We develop discrete
$W^2_p$-norm error estimates
for the Oliker-Prussner method applied
to the Monge-Amp\`ere equation.
This is obtained by extending discrete
Alexandroff estimates and showing that 
the contact set of a nodal function
contains information on its second order
difference.  In addition, we show
that the size of the complement
of the contact set is controlled
by the consistency of the method.
Combining both observations, we show that the error estimate
\[
\|u - u_h\|_{W^2_p(\Nhi)} \leq
C \begin{cases}
  h^{1/p}   \quad &\mbox{if $p > d$,}
\\
h^{1/d} \big(\ln\left(\frac 1 h \right)\big)^{1/d}  \quad &\mbox{if $p \le d$,}
\end{cases}
\] 
where the constant $C$ 
depends on $\|{u}\|_{C^{3,1}(\bar\dm)}$, the dimension $d$, and the constant $p$.
Numerical examples are given in two space dimensions and confirm that the estimate is sharp in several cases. 
\end{abstract}

\pagestyle{myheadings}
\thispagestyle{plain}
\markboth{M. Neilan and W. Zhang}{$W^2_p$ estimate of the \MA equation}

%
\section{Introduction}
%
\thispagestyle{empty}
In this paper we develop 
discrete $W^2_p$ error estimates 
for numerical approximations of 
 the \MA equation
with Dirichlet boundary conditions:
\begin{subequations}\label{MA}
\begin{alignat}{2}\label{MA1}
\det(D^2 u) & = f\qquad \text{in }\Omega,\\
\label{MA2}
u &=0 \qquad \text{on }\p\Omega,
\end{alignat}
\end{subequations}
with given function $f\in C(\bar\Omega)$
satisfying $\underline{f} \le f\le \bar{f}$ in $\bar\Omega$,
for some positive constants $\underline{f},\bar{f}$.
Here, $D^2 u$ denotes the Hessian matrix of $u$.
The domain $\Omega\subset \mathbb{R}^d$ 
is assumed to be bounded and uniformly convex.
We seek a solutions to \eqref{MA}
in the class of convex functions, 
which ensures ellipticity of
the problem and its unique solvability \cite{Gutierrez01}.

The method  we analyze in this paper 
is due to Oliker and Prussner \cite{OlikerPrussner88}, which
is based  on a geometric notion of generalized
solutions called Alexandroff solutions.  In this setting,
the determinant of the Hessian matrix of $u$
in \eqref{MA1} is interpreted as the 
measure of the sub-differential of $u$; see \cite{Gutierrez01}.
The method proposed
in \cite{OlikerPrussner88} simply poses this solution
concept onto the space of nodal functions
and enforces the geometric condition
implicitly given in \eqref{MA1} at a finite number
of points.  Namely, the method seeks a nodal function $u_h$
satisfying the Dirichlet boundary conditions on boundary nodes, and
\[
|\p u_h(x_i)| = f_i
\]
at all interior grid points $x_i$.  Here,  $\p u_h(x_i)$ denotes
the sub-differential of $u_h$ at $x_i$, $|\cdot|$ is
the $d$-dimensional Lebesgue measure, 
 $f_i\approx h^d f(x_i)$, and $h$ is the mesh parameter.
Existence and uniqueness of the method, and convergence
to the Alexandroof solution is shown in two dimensions in \cite{OlikerPrussner88}.

Recently, Nochetto and the second author
derived pointwise error estimates of the Oliker-Prussner
scheme \cite{NochettoZhang17}.  There it is shown that, if
the exact convex solution to \eqref{MA} is sufficiently
smooth, and if the nodes are translation invariant, 
then the error is of (optimal) order $\mathcal{O}(h^2)$
in the $L_\infty$ norm.  Generalities of these results, depending
on solution regularity, are also given.
The main tools
to develop these results include operator consistency estimates,
the Brunn-Minkowski inequality, 
and discrete 
Alexandroff-Bakelman-Pucci
estimates for continuous, piecewise linear functions \cite{KuoTrudinger00,NochettoZhang17A}.

Our contribution in this paper is to extend these results and to develop
discrete $W^2_p$ error estimates for all $p\in [1,\infty)$.
 To summarize this result,  we first introduce a discrete $W^2_p$ norm for discrete nodal functions. We
define the second-order difference
operator of a nodal or continuous function $v$
in the direction $e\in \bbZ^d$ at a node $x_i$ as
\begin{align*}
\delta_e v(x_i):= \frac{v(x_i+ h e) -2 v(x_i)+v(x_i-he)}{|e|^2 h^2},
\end{align*}
where $|e|$ denotes the Euclidean norm
of $e$, and it is assumed that $x_i\pm h e$ is also a node in the domain $\bar{\dm}$.
If either $x_i - h e$ or $x_i + h e$ is outside $\dm$, we define 
\begin{align*}
\delta_e v(x_i):= \frac{\rho_2 v(x_i+ \rho_1 h e) - (\rho_1 + \rho_2) v(x_i)+ \rho_1 v(x_i- \rho_2 he)}{\rho_1 \rho_2 (\rho_1 + \rho_2) |e|^2 h^2 /2},
\end{align*}
where $\rho_1$ and $\rho_2$ are the largest number in $(0,1]$ such that $x_i+ \rho_1 h e$ and $x_i- \rho_2 h e$ are in $\bar{\Omega}$, respectively.
The (weighted) $W^2_p$-norm of a nodal function $v$ with respect
to direction $e$ on a set of nodes $S$ is given by
\begin{align*}
\|v\|_{W^2_{p}(S)} := \Big(\sum_{x_i\in S} f_i |\delta_e v(x_i)|^p\Big)^{1/p}.
\end{align*}
The main result of the paper, precisely given in Theorem \ref{thm:W2d}, is the estimate
\[
\|N_h u- u_h\|_{W^2_p(\Nhi)}\leq 
\begin{cases}
  C h^{1/p}   \quad &\mbox{if $p>d$,}
\\
C h^{1/d} \ln\left(\frac 1 h \right)^{1/d}  \quad &\mbox{if $p\leq d$,}
\end{cases}
\] 
where $N_h u$ denotes the nodal interpolant of $u$.
 Similar 
to the arguments in \cite{NochettoZhang17},
one of the tools we use is operator consistency of the method.
In addition, we extend the discrete
Alexandroff-Bakelman-Pucci estimates
given in \cite{KuoTrudinger00,NochettoZhang17A},
and show that the contact set also contains
useful information about the second-order differences.

Because of its wide array
of applications in e.g., differential geometry,
optimal mass transport, and meteorology,
several numerical methods
have been developed for the Monge-Amp\`ere problem.
These include the monotone finite difference
schemes \cite{Oberman08,FroeseOberman11,Benamou16,Mirebeau15},
the vanishing moment method \cite{FengNeilan09}, $C^1$ finite element
methods \cite{Bohmer08,Awanou15b},
$C^0$ penalty methods \cite{Brenner11,Neilan14,Awanou15},
and semi-Lagrangian schemes \cite{FengJensen17}.
We also refer the interested reader to a review of numerical methods for fully nonlinear elliptic equations \cite{NeilanSalgadoZhang}.
One application of our results is
to feed the solution of the Oliker-Prussner method
into a higher-order scheme.  For example, 
the results given in \cite{Neilan14} state
that Newton's method converges
to the discrete solution provided that difference
between the initial guess and the exact
solution is sufficiently small in a $W^2_p$-norm.
Therefore, we show that the solution of the 
Oliker-Prussner scheme can be used as an initial guess
within a higher-order scheme.
We will explore this idea in a coming paper.

The organization of the paper is as follows.
In the next section, we state the Oliker-Prussner method
and state some preliminary results.
In Section \ref{sec:Consistency}
we give operator consistency results
of the scheme.  Section \ref{sec:Stability} gives 
stability results with respect to the second-order difference operators,
and in Section \ref{S:W21estimate}
we provide $W^{2}_p$ error estimates.
Finally, we end the paper with some numerical
experiments in
Section \ref{S:numerics}.

%
\section{Preliminaries}
%
\subsection{Nodal Set and Nodal Function}
Let $\Nh$ be a set of nodes in the domain $\bar\dm$. We denote the set of interior nodes $\Nhi := \Nh \cap \dm$,
the set of boundary nodes 
$\Nhb := \Nh \cap \bdry$, and the nodal set 
\[
\Nh = \Nhi \cup \Nhb.
\]
To ensure that the interior node is not too close to the boundary $\partial \dm$, we require that 
\begin{align}\label{nodalassumption}
\textrm{dist}(z, \partial \dm) \geq \frac h 2 \quad \mbox{for any nodes $z \in \Nhi$}
\end{align}
Such a nodal set can be obtained by removing the nodes whose distance to $\partial \Omega$ is less than $h/2$. 
We assume that the nodal set is {\it translation invariant}, i.e., 
there exist a point $b \in \bbR^d$ and a basis $\{ e_i \}_{i=1}^d$ 
in $\bbR^d$ such that any interior node $z \in \Nhi$ can be written as
\begin{align}\label{translationinvariant}
  z = b + \sum_{i=1}^d h z_i e_i 
  \quad \mbox{for some integers $z_i \in \mathbb{Z}$.}
\end{align}
Since the basis $e_i$ can be transformed into the canonical basis in $\mathbb{R}^d$ under a linear transformation, hereafter to simplify the presentation, we will assume that 
$\Nhi = b + h \mathbb Z ^d $. 
We also make the following additional assumption on the boundary nodal set $\Nhb$:
\begin{align}
  \dist(x,\Nhb) \leq h,\qquad  \forall x\in \p \Omega.
\end{align} 
We say the nodal spacing of $\Nh$ is $h$.
It is worth mentioning that one can construct a translation invariant $\Nh$ on a curved domain $\dm$.
In fact, for a nodal set $\Nh$ to be translation invariant, we only require the interior nodal set $\Nhi$ satisfies \eqref{translationinvariant}, while no such requirement is made on the boundary nodes.

Associated with the nodes 
is a simplicial triangulation $\mathcal{T}_h$,
with vertices $\Nh$. We denote
by $h_T$ the diameter of $T\in \mct$,
and by $\rho_T$ the diameter of the largest
inscribed ball in $T$.  We assume that  
that the triangulation is shape-regular, i.e., there
exists $\sigma>0$ such that
\begin{align*}
\frac{h_T}{\rho_T}\le \sigma\qquad \forall T\in \mct.
\end{align*}

We denote by $\{\phi_i\}_{i=1}^n$,
with $n = \# \Nhi$, the canonical piecewise linear
hat functions associated with $\mathcal{T}_h$.
Namely, the function $\phi_i\in C(\bar\Omega)$
is a piecewise linear polynomial with respect to $\mathcal{T}_h$,
and is uniquely determined by the condition $\phi_i(x_j) = \delta_{i,j}$ (Kronecker delta)
for all $x_j\in \Nhi$ and $\phi_i(x_j)=0$ for all $x_j\in \Nhb$.
We denote by $\omega_i$ the support of $\phi_i$, i.e.,
the patch of elements in $\mathcal{T}_h$ that have $x_i$
as a vertex.

A function defined on $\Nh$ is called a nodal function,
and we denote the space of nodal functions by $\mathcal{M}_h$.
For a nodal function $g$ with nodal value $\{g_i\}_{x_i\in \Nh}$,
and for a subset of nodal points $\calC \subset \Nh$,
we set the discrete $\ell^d$ norm 
as
\begin{align*}
  \|g\|_{\ell^d(\calC)}:=\Big(\sum_{x_i\in \calC} |g_i|^d \Big)^{1/d}.
\end{align*}
%
%
We say that a nodal function $u_h\in \mathcal{M}_h$ is convex 
if, for all $x_i\in \Nhi$, there exists a supporting hyperplane
$L$ of $u_h$, i.e.,
\[
L(x_j)\le u_h(x_j)\quad \forall x_j\in \Nh\text{ and } L(x_i) = u(x_i).
\]
The convex envelope of $u_h$ is
the function $\Gamma (u_h)\in C(\bar\Omega)$ given by
\[
\Gamma(u_h)(x) = \sup_{L} \{L(x)\text{ is affine}:\ L(x_i)\le u_h(x_i)\ \forall x_i\in \Nh\}.
\]
Finally, we denote by $N_h:C(\bar\Omega)\to \mathcal{M}_h$
the nodal interpolant satisfying $N_h v(x_i) = v(x_i)$
for all $x_i\in \Nh$.  It is easy to see
that if $v$ is a convex function on $\bar\Omega$,
then $N_h v$ is a convex nodal function.

\subsection{The Oliker-Prussner Method}\label{sec:Method}

\begin{figure}
\begin{center}
  \includegraphics[scale=0.2]{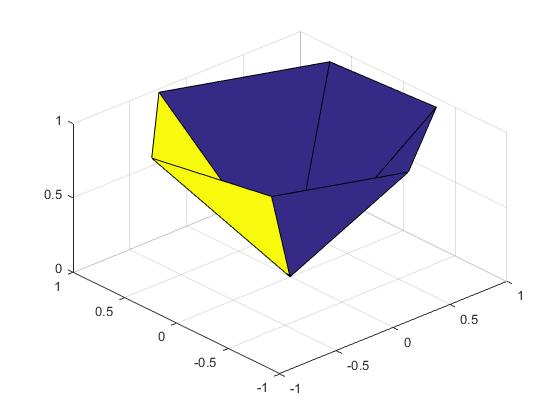} 
  \includegraphics[scale=0.2]{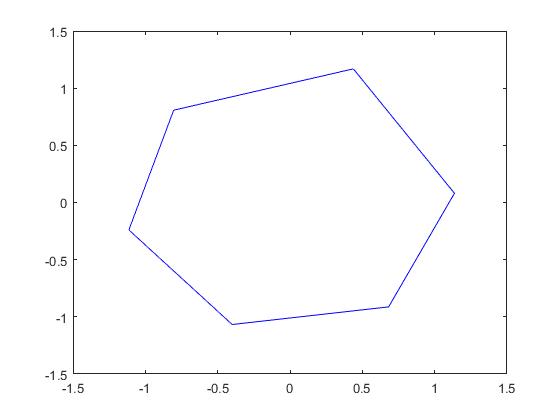} 
\\
  \includegraphics[scale=0.2]{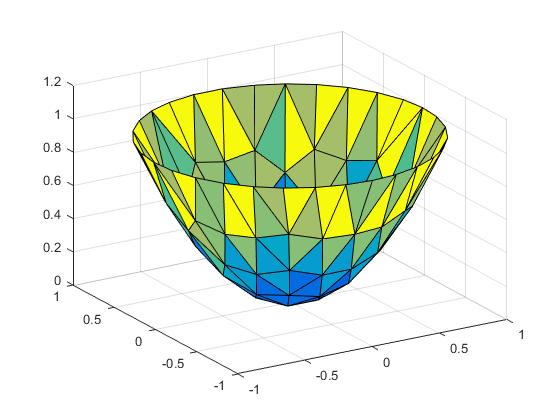} 
  \includegraphics[scale=0.2]{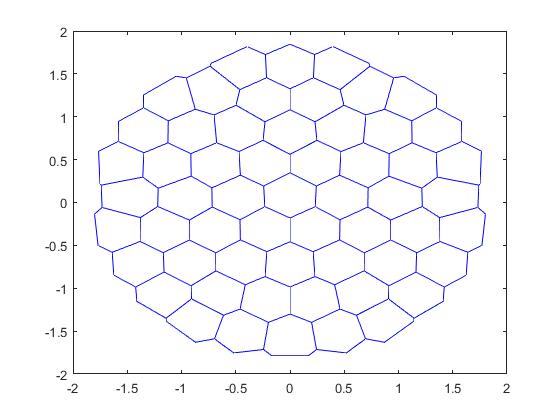} 
\caption{
A convex nodal function $u_h$ induces a convex piecewise linear function $\gamma_h=\Gamma(u_h)$.
The sub-differential $\partial u_h(0)$ of the convex
nodal function $u_h$ at node $0$ is the convex hull of the piecewise gradients $\nabla \gamma_h|_{T}$, which is the polygon in the second figure. 
Let the domain $\dm$ be a unit ball centered at $0$ and $\Nh$ be a nodal set in $\dm$. 
A convex nodal function $u_h$ defined on $\Nh$ induces a piecewise linear function $\Gamma(u_h)$.
For each node $x_i\in \Nh$, there is an associated subdifferential $\partial u_h(x_i)$ which corresponds to a polygon cell in the last figure. 
The piecewise gradient of $u_h$ can be viewed as a map between the domain $\dm$ and the diagram.
}
\label{fig:subdifferential}
\end{center}
\end{figure}

To motivate the method introduced
in \cite{OlikerPrussner88}, we first
introduce the notion of an Alexandroff solution
to the \MA equation \eqref{MA}.
To this end, note that if the solution to \eqref{MA}
is strictly convex, and if $u\in C^2(\Omega)$,
then a change of variables reveals  that
\[
\int_E f\, dx = \int_E \det(D^2 u)\, dx = \int_{\nab u(E)} dx = |\nab u(E)|\quad \text{for all Borel }E\subset \Omega,
\]
where $|\nab u(E)|$ denote the $d$-dimensional Lebesgue measure of $\nab u(E)
=\{\nab u(x):\ x\in E\}$.
To extend this identity to a larger class of functions, 
we introduce the subdifferential 
of the function $u$ at the point $x_0$ as
\begin{align*}
  \p u(x_0) = \{p\in \bbR^d:\ u(x)\ge u(x_0)+p\cdot (x-x_0) \quad \forall x \in \dm\}.
\end{align*}
Thus, $\p u(x_0)$ is the set of supporting hyperplanes
of the graph of $u$ at $x_0$.  If 
$u$ is strictly convex and smooth then $\p u(x_0) = \{\nab u(x_0)\}$,
and the same calculation as above shows that
\begin{align}
\label{eqn:AlexMotivation}
\int_E f\, dx =  |\p  u(E)|\quad \text{for all Borel }E\subset \Omega.
\end{align}
\begin{definition}
A convex function $u\in C(\bar\Omega)$
is an {\em Alexandroff  solution} to \eqref{MA}
provided that $u=0$ on $\p\Omega$
and \eqref{eqn:AlexMotivation} is satisfied.
\end{definition}

The method
introduced in \cite{OlikerPrussner88}
simply poses this solution concept
onto the space of nodal functions.
To do so, the definition of the subdifferential
is extended to the spaces of nodal functions
in the natural way:
\begin{align}\label{def:subdifferential}
\p u_h(x_i) = \{p\in \mathbb{R}^d:\ u(x_j)\ge u_h(x_i)+p\cdot (x_j-x_j)\ \forall x_j\in \Nh\}.
\end{align}

To characterize the sub-differential of a nodal function $u_h$, we note that the convex envelope of a convex nodal function $u_h$, which is a piecewise linear function defined in $\dm$, induces a mesh $\tilde{\mathcal{T}}_h$; see Figure \ref{fig:subdifferential}.
Then the sub-differential of $u_h$ at node $x_i$ can be characterized as the convex hull of the constant gradients
$
\gradv \Gamma(u_h)|_T $
for all $T \in \tilde{\mathcal{T}}_h$ which contain $x_i$; see Figure \ref{fig:subdifferential}.

The discrete method
is to find a convex nodal 
function $u_h$ with $u_h=0$
on $\Nhb$ and
\begin{align}\label{eqn:OPMethod}
|\p u_h(x_i)| = f_i\qquad \forall x_i\in \Nhi,
\end{align}
where
\begin{align}\label{eqn:fiDef}
f_i = \int_{\Omega} f(x) \phi_i(x)\, dx = \int_{\omega_i} f(x) \phi_i(x)\, dx.
\end{align}
\begin{remark}
Existence and uniqueness 
of a solution to \eqref{eqn:OPMethod}
is given in {\rm\cite{OlikerPrussner88,NochettoZhang17}}.
\end{remark}


\subsection{Brunn Minkowski inequality and subdifferential of convex functions}\label{sub:BM}

In this subsection, we develop a few techniques which will be useful in establishing the error estimate.
We start with the celebrated Brunn Minkowski inequality which relates the volumes of compact sets of $\mathbb R^d$.

\begin{prop}[Brunn Minkowski inequality]\label{BM}
  Let $A$ and $B$ be two nonempty compact subsets of $\mathbb R^d$ for $d \geq 1$. Then the following inequality holds:
\[
  |A + B|^{1/d} \ge |A|^{1/d} + |B|^{1/d},
\]
where $A+B$ denotes the Minkowski sum:
\[
  A + B := \{ v + w \in \mathbb{R}^d:  v \in A \text{ and } w \in B \}.
\]
\end{prop}


Next, we make the following observation on the sum of two subdifferential sets. 
\begin{lemma}[Lemma 2.3 in \cite{NochettoZhang17}]\label{lem:addSubDiff}
Let $u_h$ and $v_h$
be two convex nodal functions.
Then there holds
\begin{align*}
\p u_h(x_i) + \p v_h(x_i)\subset \p (u_h+v_h)(x_i)
\end{align*}
for all $x_i \in \Nhi$. 
\end{lemma}
\begin{proof}
Let $p_1$ and $p_2$ be in $\partial u_h(x_i)$ and $\partial v_h(x_i)$, respectively. By the definition of subdifferential \eqref{def:subdifferential}, we have 
\begin{align*}
  p_1 \cdot (x_j - x_i) \leq & u_h(x_j) - u_h(x_i) \quad \forall x_j \in \Nh,
  \\
  p_2 \cdot (x_j - x_i) \leq & v_h(x_j) - v_h(x_i) \quad \forall x_j \in \Nh.
\end{align*}
Adding both inequalites, we obtain 
\[
  (p_1 + p_2) \cdot (x_j - x_i) \leq  (u_h + v_h) (x_j) - (u_h + v_h)(x_i) \quad \forall x_j \in \Nh.
\]
This shows that $p_1 + p_2 \in \p (u_h + v_h)(x_i)$.
\end{proof}

Combining both estimates, we derive the following result.
\begin{lemma}\label{lem:convexity_subdifferential}
Let $u_h$ and $v_h$ be two convex nodal functions defined on $\Nh$ and $\calC_h$ be the lower contact set of $(u_h - v_h)$:
\[
\calC_h:=\big\{x_i\in \Nhi:\ \Gamma(u_h-v_h)(x_i) = (u_h-v_h)(x_i)\big\}.
\]
 Then for any node $x_i \in \calC_h$,
\begin{align}\label{convexity_subdifferential}
  |\p \Gamma (u_h - v_h)(x_i)|^{1/d} \leq  |\p u_h(x_i)|^{1/d} -  |\p v_h(x_i)|^{1/d}. 
\end{align}
\end{lemma}
\begin{proof}
The proof of this result
is implicitly given in \cite[Proposition 4.3]{NochettoZhang17},
but we give it here for completeness.

The definition of the convex envelope 
and the subdifferential shows that
\begin{align*}
\p \Gamma (u_h-v_h)(x_i)\subset \p (u_h-v_h)(x_i)
\end{align*}
for all $x_i\in \Ceb$. Applying Lemma \ref{lem:addSubDiff}
then yields
\begin{align*}
\p v_h(x_i) + \p \Gamma(u_h-v_h)(x_i)\subset \p v_h(x_i)+\p (u_h-v_h)(x_i) \subset \p u_h(x_i).
\end{align*}
An application of the Brunn-Minkowski inequality (cf.~Lemma \ref{BM})
gets
\begin{align*}
|\p v_h(x_i)|^{1/d} + |\p \Gamma (u_h-v_h)(x_i)|^{1/d}
&\le |\p v_h(x_i)+ \p \Gamma(u_h-v_h)(x_i)|^{1/d}\\
&\le |\p u_h(x_i)|^{1/d}.
\end{align*}
Rearranging terms we obtain \eqref{convexity_subdifferential}.
\end{proof}

%

We also note that the numerical method \eqref{eqn:OPMethod} has a discrete comparison principle. Here, we refer to \cite{NochettoZhang17} for a  proof.

\begin{lemma}[discrete comparison principle, Corollary 4.4 in \cite{NochettoZhang17}]\label{lem:discrete_compare}
Let $v_h,w_h\in \mathcal{M}_h$ satisfy $v_h(x_i) \ge w_h(x_i)$
for all $x_i\in \Nhb$ and $|\p v_h(x_i)|\le |\p w_h(x_i)|$
for all $x_i\in \Nhi$.  Then
\[
v_h(x_i)\ge w_h(x_i)\qquad \forall x_i\in \Nh.
\]
\end{lemma}

%
\section{Consistency of the Oliker-Prussner method}\label{sec:Consistency}
%
In this section, we state
the consistency of the method 
\eqref{eqn:OPMethod} given in \cite[Lemma 5.3, Proposition 5.4]{NochettoZhang17}.
The result shows that
the relative consistency error
is of order $\mathcal{O}(h^2)$ away from
the boundary and of order $\mathcal{O}(1)$
in a $\mathcal{O}(h)$ region of the boundary.
%
\begin{lemma}\label{lem:QuadraticConsistency}
Let $\Nh$ be translation invariant nodal set defined on the domain $\Omega$.
If  $u\in C^{k,\alpha}(\bar{\Omega})$
is a convex function with $0 < \lambda I \le D^2 u\le \Lambda I$
and $2\le k+\alpha\le 4$, 
there holds, for ${\rm dist}(x_i,\p \Omega) \ge Rh$,
\begin{align}
\label{eqn:Wishful}
\big| |\p N_h u(x_i)| - f_i\big|
\le C h^{k+\alpha+d-2},
\end{align}
where $R$ depends on $\lambda$ and $\Lambda$.
Moreover, there holds for ${\rm dist}(x_i,\p \Omega)\le Rh$,
\[
\big|\p N_h u(x_i)-f_i\big|\le C h^d.
\]
\end{lemma}

\begin{remark}
The regularity of $f$ and $\p\Omega$,
the strict convexity of $\Omega$,
and the positivity of $f$ guarantees
that the convex solution
to \eqref{MA} enjoys the regularity $u\in C^{k,\alpha}(\bar{\Omega})$.
For example, if $f\in C^{k-2,\alpha}(\bar\Omega)$
and $\Omega$ is smooth,
then the solutions satisfies $u\in C^{k,\alpha}(\bar\Omega)$ \cite{Gutierrez01,CafNirSpruck84,TrudingerWang08} 
\end{remark}

Thanks to the consistency error of the method, Lemma \ref{lem:QuadraticConsistency}, an $L_{\infty}$-error estimate is derived in \cite{NochettoZhang17} which states
\begin{proposition}\label{prop:LinftyEstimate}
Let $\Omega$ be uniformly convex
and $\Nhi$ be translation invariant.  
Suppose further that the boundary nodes satisfy \eqref{nodalassumption},
that $f\ge \underline{f}>0$, and that  the convex solution to \eqref{MA}
satisfies $u\in C^{k,\alpha}(\bar\Omega)$
for some $2\le k+\alpha\le 4$ and $0< \lambda I \leq D^2 u \leq \Lambda I$.  Then the numerical solution to the discrete \MA equation \eqref{eqn:OPMethod} satisfies
\begin{align*}
\|u_h-N_h u \|_{L_\infty(\Nh)}\le C h^{k+\alpha-2} \|u\|_{C^{k,\alpha}(\bar\Omega)},
\end{align*}
where
$
\|v_h\|_{L_\infty(\Nh)}:=\max_{x_i\in \Nh} |v_h(x_i)|.
$
\end{proposition}

We note that if $u \in C^{3,1}(\dm)$, then the optimal order of the $L_\infty$ error is $\mathcal{O}(h^2)$. 
By this $L_{\infty}$ error estimate and the assumption \eqref{nodalassumption} that the boundary node is at least $h/2$ away from the boundary, we immediately deduce that 
$
| \delta_e (N_h u - u_h) (x_i) | 
$
is bounded.
This observation will be useful in the following sections when we investigate the discrete $W^2_p$ error estimate. 

%
\section{Stability of the Oliker-Prussner method}\label{sec:Stability}
%

To derive the discrete $W^2_p$-estimate, we first make an observation that the contact 
set of a nodal function contains interesting information on its second order difference.

\begin{lemma}[estimate of second order difference]\label{lem:bound}
Given two convex nodal functions $v_h$ and $u_h$ defined on the nodal set $\Nh$, let 
\[
\whb = u_h - (1-\epsilon)v_h
\quad \mbox{and} \quad
\wha = v_h - (1-\epsilon)u_h
\]
for some $0< \epsilon \leq 1$
and the contact sets 
\begin{align}
\label{eqn:CebDef}
\Ceb := &\; \{ x_i \in \Nh, \quad \whb(x_i) = \Gamma \whb(x_i) \},
\\
\label{eqn:CeaDef}
\Cea := &\; \{ x_i \in \Nh, \quad \wha(x_i) = \Gamma \wha(x_i) \}.
\end{align}
If a node $x_i \in \Ceb \cap \Cea$, then  
\begin{align}\label{2ndorderestimate}
-\epsilon \delta_e v_h(x_i) \leq \delta_e (u_h - v_h) (x_i) \leq \frac{\epsilon}{1 - \epsilon}  \delta_e v_h (x_i)
\end{align}
for any vector $e \in \mathbb{Z}^d$.
\end{lemma}
\begin{proof}
We observe that if a node is in the contact set $x_i \in \Ceb$, then the second order difference
of $\whb$ satisfies $\delta_e \whb(x_i) \ge \delta_e \Gamma \whb(x_i)  \ge 0$
for any vector $e\in \bbZ^d$.
Hence, for any node $x_i \in \Ceb$, we have
\begin{align}\label{lowerbound}
\delta_e (u_h - v_h) (x_i) \geq -\epsilon  \delta_e v_h (x_i). 
\end{align}
This inequality yields a lower bound of the second order difference. 

To derive the upper bound, we apply a similar argument above to the function $\wha$ and derive 
\[
\delta_e (v_h - u_h) (x_i) \geq -\epsilon  \delta_e u_h (x_i)
\]
for any node $x_i \in \Cea$.
A simple algebraic manipulation yields
\begin{align}\label{upperbound}
\delta_e (u_h - v_h) (x_i) \leq \frac{\epsilon}{1 - \epsilon}  \delta_e v_h (x_i). 
\end{align}
Combining both the lower bound \eqref{lowerbound} and upper bound \eqref{upperbound}, we obtain the desired estimate.
\end{proof}

\begin{remark}\label{rmk:observation}
The lemma above shows that we have control of the error $\delta_e (u_h - v_h)$ on the 
contact sets $\Ceb$ and $\Cea$. 
Define the set $E_{\tau}$ to be
\begin{align}\label{Et}
E_\tau = \left \{ x_i \in \Nh, \ \delta_e (v_h - u_h)(x_i) \geq \tau \delta_e v_h(x_i) 
\quad \mbox{for some vector $e \in \mathbb{Z}^d$} \right\},
\end{align}
where $\tau = \epsilon/(1 - \epsilon)$. 
Then the proof of Lemma \ref{lem:bound} shows 
that $E_\tau$ is contained in the non-contact set
\begin{align}
\label{eqn:Seps}
S_\epsilon := \Nh \setminus  \Ceb.
\end{align}
Analogously,
\begin{align*}
E^\tau :&= \left\{x_i\in \Nh,\ \delta_e(u_h-v_h)(x_i)\geq \tau \delta_e v_h(x_i)
\quad \mbox{for some vector $e \in \mathbb{Z}^d$} \right\}\\
&\subset S^\epsilon:=\Nh\setminus  \Cea. 
\end{align*}
\end{remark}


In the next step, we estimate the cardinality of $S_\epsilon$. 
Heuristically, if $\epsilon = 1$, then $\whb = u_h$ 
which is a convex nodal function,  and so we have $S_\epsilon = \emptyset$. 
As $\epsilon$ decreases to zero, the function $\whb$ 
becomes `less convex', and the cardinality $\#(S_\epsilon)$ increases;
see Figure \ref{fig:Remark1}.
Therefore, our next goal is to estimate how fast $\#(S_\epsilon)$ increases as $\epsilon \to 0$. 
The following lemma shows that this is controlled by the consistency error of the method.

\begin{figure}
\begin{center}
  \begin{tikzpicture}
\draw  [gray, thin] (1,0) -- (6,0);
\draw  [blue, very thick, domain=1:6] plot (\x, { (0.5*\x-0.5) * (0.5*\x-3) });
\draw (5, 1)  [above left]  node {$w_\epsilon(x) = u_h(x)$};
\draw  [gray, thin] (7,0) -- (12,0);
\draw  [blue, very thick, domain=7:12] plot (\x, {0.5* (0.5*\x-3.5) *(0.5 *\x - 4)* (0.5*\x-6) });
\draw  [dashed, very thick](7,0) -- (10, -0.75);
\draw [dashed, very thick, domain=10:12] plot (\x, {0.5* (0.5*\x-3.5) *(0.5 *\x - 4)* (0.5*\x-6) });
\draw (11, 1)  [above left]  node {$w_\epsilon(x) = u_h - \frac 12 v_h$};
\draw  (10,-1) [below left]  node {$\Gamma w_\epsilon(x)$};
\end{tikzpicture}
\end{center}
\caption{A pictorial description of Remark \ref{rmk:observation}}
\label{fig:Remark1}
\end{figure}
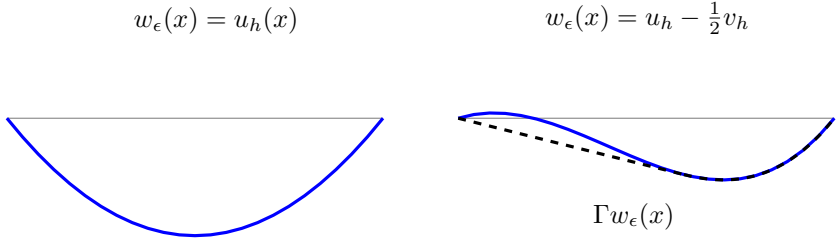


\begin{prop}\label{prop:measureestimate}
Let $u_h$ and $v_h$ be two convex nodal functions satisfying $u_h=v_h$ on $\Nhb$,
$u_h \le v_h$ in $\Nhi$, and
\begin{align}\label{eqn:uhvh}
| \partial u_h(x_i) | = f_i,
\quad \mbox{ and } \quad
| \partial v_h(x_i) | = g_i
\end{align}
for all $x_i \in \Nhi$.
For any subset $S \subset \Nhi$, let  
\begin{align}
\label{eqn:MuDef}
\mu (S) = \sum_{x_i \in S} f_i
\quad \mbox{and} \quad
\nu_{\tau} (S) = \sum_{x_i \in S} \left( f_i^{1/d} + \frac 1 {\tau} e_i^{1/d} \right)^d,
\end{align}
where $e_i^{1/d} = g_i^{1/d} - f_i^{1/d}$. 
Then 
\begin{align}\label{eqn:measureestimate}
\mu(S_\epsilon) \leq \nu_{\tau}(\Ceb) - \mu(\Ceb),
\end{align}
where $\Ceb$ is given by \eqref{eqn:CebDef},
$S_\epsilon$ is given by \eqref{eqn:Seps},
and $\tau = \epsilon/(1-\epsilon)$.  Consequently, there holds
\begin{align}\label{eqn:measureestimate2}
\mu(S_\epsilon)\le \tau^{-1} C_f \|e^{1/d}\|_{\ell^d(\Ceb)},
\end{align}
with $C_f = d \|f^{1/d}\|_{\ell^d(\Nhi)}^{d-1}.$
\end{prop}
\begin{proof}
We first show that 
\begin{equation}\label{eqn:WTS123}
\sum_{x_i\in \Nhi} \epsilon \partial  u_h (x_i) \subset \sum_{x_i\in \Nhi} \partial \Gamma \whb (x_i),
\end{equation}
where $\whb = u_h -(1-\epsilon)v_h$.
Since $u_h \leq v_h$ in $\Nhi$ and $u_h = v_h$ on $\Nhb$, we get
\[
  \whb \leq \epsilon u_h  \mbox{ in $\Nhi$, }\quad \text{and} \quad
  \whb = \epsilon u_h  \mbox{ on $\Nhb$.}
\]
Taking convex envelope on both side of the inequality, we obtain
\begin{align}
\label{eqn:WTSABC}
  \Gamma \whb(x) \leq \epsilon \Gamma u_h(x) \quad \mbox{in $\Omega$ and} \quad
  \Gamma \whb(x) = \epsilon \Gamma u_h(x) \quad \mbox{on $\bdry$.}
\end{align}
Since $u_h = \Gamma u_h$ on $\Nh$ due to the convexity of $u_h$, the inequality \eqref{eqn:WTSABC} implies 
\eqref{eqn:WTS123}. 

Taking measure on both sides of \eqref{eqn:WTS123} and substituting \eqref{eqn:uhvh} yields
\[
  \epsilon^d \sum_{x_i \in \Nhi} f_i = \epsilon^d \sum_{x_i \in \Nhi} |\partial u_h(x_i)| \leq   \sum_{x_i \in \Ceb} |\partial \Gamma \whb(x_i)|.
\]
In view of the convexity of the measure of the subidfferential \eqref{convexity_subdifferential}, 
\[
  |\partial \Gamma \whb(x_i)|^{1/d} \leq |f^{1/d}_i - (1 - \epsilon) g^{1/d}_i |.
\]
Therefore, we infer that
\[
  \epsilon^d \mu(\Nhi) = \epsilon^d \sum_{x_i \in \Nhi} f_i \leq   \sum_{x_i \in \Ceb} |f^{1/d}_i - (1 - \epsilon) g^{1/d}_i |^d .  
\]
Thus, {subtracting $\epsilon^d \mu(\Ceb)$, we obtain}
\begin{align*}
  \epsilon^d \mu(S_{\epsilon}) =
  \epsilon^d \sum_{x_i \in S_\epsilon} f_i 
  \leq &\; \sum_{x_i \in \Ceb}  \big(|\epsilon f^{1/d}_i + (1 - \epsilon) e_i^{1/d}  |^d - \epsilon^d f_i \big).
\end{align*}
Therefore, dividing $\epsilon^d$, we obtain
\begin{align*}
\mu(S_\epsilon)  \leq  \nu_{\tau}(\Ceb)  - \mu(\Ceb).
\end{align*}

To derive the estimate \eqref{eqn:measureestimate2}, we first
see that \eqref{eqn:measureestimate} is equivalent to 
\begin{align*}
\|f^{1/d} \|_{\ell^d(\Nhi)} \le \| f^{1/d} +  \tau^{-1} e^{1/d} \|_{\ell^d(\Ceb)},
\end{align*}
and therefore $\|f^{1/d}\|_{\ell^d(\Nhi)} - \|f^{1/d}\|_{\ell^d(\Ceb)}\le \tau^{-1} \|e^{1/d}\|_{\ell^d(\Ceb)}$
by the Minkowski inequality.  From this estimate and the inequality $a^d - b^d\le d a^{d-1}(a-b)$ for $a\ge b$,
we derive
\begin{align*}
\mu(S_\epsilon) 
&= \|f^{1/d}\|_{\ell^d(\Nhi)}^d - \|f^{1/d}\|_{\ell^d(\Ceb)}^d\\
&\le d \|f^{1/d}\|_{\ell^d(\Nhi)}^{d-1}\big(\|f^{1/d}\|_{\ell^d(\Nhi)} -  \|f^{1/d}\|_{\ell^d(\Ceb)}\big)\\
&\le C_f \tau^{-1} \|e^{1/d}\|_{\ell^d(\Ceb)}.
\end{align*}\hfill
\end{proof}

%
           \section{$W^2_p$-estimate of the method}\label{S:W21estimate}
%

To establish
$W^{2}_p$-estimates of the method,
we first introduce an estimate of the discrete $L_1$ norm of a nodal function 
in terms of its level sets.

\begin{lemma}\label{lem:L1norm}
Let $s_h$ be a bounded nodal function with 
$|s_h(x_i)| \leq M$ for some $M >0$.
Then, for any $\sigma>0$,
\[
\sum_{x_i\in \Nhi} f_i |s_h(x_i)|   \leq \sigma \sum_{k=0}^N \mu(A_k),
\]
where
\[
 A_k := \{x_i \in \Nhi:\ |s_h(x_i)| \geq k \sigma\},
\]
$\mu(\cdot)$ is given by \eqref{eqn:MuDef}, and $N = \ceil{M / \sigma}$.
\end{lemma}

\begin{proof}
The estimate is illustrated in the Figure \ref{fig:L1norm}. Here, we give a rigorous proof.
\\
Set
\begin{align*}
P_k:=\{x_i\in \Nhi:\ k\sigma \le |s_h(x_i)|<(k+1)\sigma\}.
\end{align*}
Then we clearly have
\begin{align*}
\sum_{x_i\in \Nhi} f_i |s_h(x_i)|  
= \sum_{k=0}^N \sum_{x_i\in P_k} f_i |s_h(x_i)|
\le \sum_{k=0}^N (k+1)\sigma \mu(P_k).
\end{align*}
We also have
\begin{align*}
A_k = \bigcup_{m\ge k}^N P_m,
\end{align*}
and so, since the sets $\{P_k\}$ are disjoint,
\begin{align*}
\mu(A_k) = \sum_{m=k}^N \mu(P_m).
\end{align*}
Therefore
\begin{align*}
\sigma \sum_{k=0}^N \mu(A_k) = \sigma \sum_{k=0}^N \sum_{m=k}^N \mu(P_m) = \sigma \sum_{k=0}^N (k+1) \mu(P_k) \ge  \sum_{x_i\in \Nhi} f_i |s_h(x_i)|.
\end{align*}
\hfill
\end{proof}

\begin{figure}
\begin{center}
\begin{tikzpicture}\label{fig:L1norm}
\draw  [gray, thin] (3,0) -- (11,0);
\draw  [->] (3,0)--(3,4) ;
\draw  [blue, very thick, domain=3:11] plot (\x, {-0.25* pow(0.5*\x-1.5, 2) * (0.5*\x-5.5) });
\draw  [dashed] (11, 0.5) -- (3,0.5) ;
\draw  [dashed] (11, 1) -- (3,1);
\draw  [dashed] (9, 2.5) -- (3,2.5);
\draw  [dashed] (4.6, 0.5) -- (4.6,-1);
\draw  [dashed] (10.7, 0.5) -- (10.7,-1);
\draw  [decorate,decoration={brace,raise=2mm,amplitude=6pt,mirror}] (4.6,-1) -- (10.7,-1) (7.7,-1.8) node {$A_1$};
\draw  [dashed] (5.4, 1) -- (5.4,0);
\draw  [dashed] (10.4, 1) -- (10.4,0);
\draw  [decorate,decoration={brace,raise=2mm,amplitude=6pt,mirror}] (5.4,0) -- (10.4,0) (7.7,-.8) node {$A_2$};
\draw  (11.2,2) [below left]  node {$s_h(x)$};
\draw  (3,0.5) [left] node {$\sigma$};
\draw  (3,1) [left] node {$2\sigma$};
\draw  (3,2.5) [left] node {$M = N \sigma$};
\draw  (7, 3.5) [above] node 
{
$
\sum_{x_i \in \Nhi} f_i |s_h| \leq \sigma \sum_{k=0}^N \mu(A_k).
$
};
\end{tikzpicture}
\caption{
  A pictorial illustration of Lemma \ref{lem:L1norm}. Here, the measure $\mu(A_k) := \sum_{x_i \in A_k} f_i$.
}
\end{center}
\end{figure}
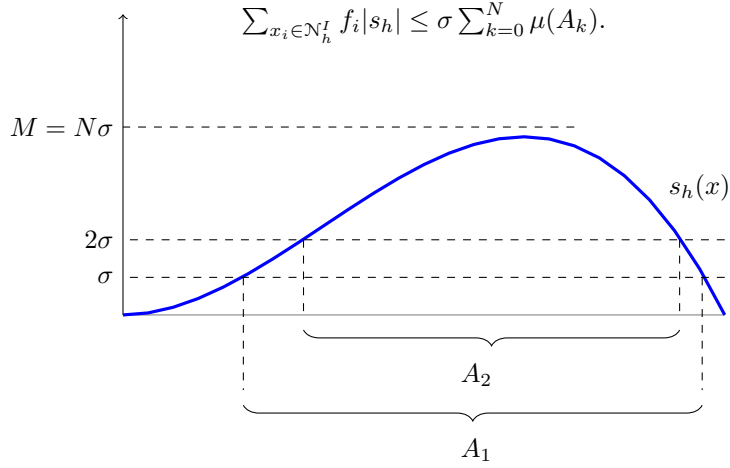

\subsection{Ideal Case}\label{subsec:ideal}
Now we are ready to prove the estimate
in the case that the consistency
error \eqref{eqn:Wishful} holds for all 
interior grid points.
%
\begin{theorem}\label{thm:C4case}
  Let $u$ be the solution of the \MA equation \eqref{MA}. Assume that 
\begin{align}\label{consistencyerror}
\big| |\partial N_h u (x_i)| - f_i \big| \leq C h^{2+d}
\quad 
\mbox{ for every node $x_i \in \Nhi$,}
\end{align}
where $N_h u$ is the interpolation of $u$ on the nodal set $\Nh$.
Assume further that $f$ is uniformly positive on $\Omega$.
Then the error in the weighted $W^2_p$-norm satisfies
\begin{align*}
\|N_h u - u_h\|_{W^2_p(\Nhi)} 
\leq C\left\{
\begin{array}{ll}
h^2 |\ln h| & \text{if $p=1$},\\
h^{2/p} & \text{if $p>1$}
\end{array}
\right.
\end{align*}\
provided that $h$ is sufficiently small.
\end{theorem}
\begin{proof}
We start by setting $v_h = (1-Ch^2)^{1/d} N_h u$, where the constant 
$C>0$ is large enough, but independent of $h$, to ensure that (cf.~\eqref{consistencyerror})
\begin{align*}
g_i := |\p v_h(x_i)| = (1-C h^2)|\p N_h u(x_i)|\le f_i.
\end{align*}
By a comparison principle (cf.~Lemma \ref{lem:discrete_compare}), we have $u_h\le v_h$ on $\Nhi$, and we see
that
\begin{align}\label{eqn:IdealConsistencyError123}
|f_i - g_i|\le C h^{2+d}\qquad \forall x_i \in \Nhi
\end{align}
 due to the assumption \eqref{consistencyerror}.
 We also have $g_i \ge C h^d$ provided $h$ is sufficiently small,
 and $|(v_h-N_h u)(x_i)|\le C h^2$.

Note that
\[
\|N_h u - u_h\|_{W^2_p(\Nhi)} \leq 
\|v_h - u_h\|_{W^2_p(\Nhi)} + C h^2 \|N_h u\|_{W^2_p(\Nhi)} 
\]
Thus, to prove the theorem, it suffices to show that 
\[
\sum_{x_i \in \Nhi} f_i | \delta_e(v_h - u_h)(x_i) |^p 
\leq C\left\{
\begin{array}{ll}
h^2 |\ln h| & \text{if $p=1$},\\
h^{2} & \text{if $p>1$}.
\end{array}
\right.
\]
Define the positive and negative parts
of $\delta_e(v_h-u_h)(x_i)$, respectively, as
\begin{align*}
\delta_e^+(v_h - u_h)(x_i) &= \max\{ \delta_e(v_h - u_h)(x_i), 0\},\\
\delta_e^-(v_h - u_h)(x_i) &= \max\{ -\delta_e(v_h - u_h)(x_i), 0\}.
\end{align*}
We shall prove 
\begin{align*}
\sum_{x_i \in \Nhi} f_i  |\delta_e^+ (v_h - u_h)(x_i)|^p \leq C\left\{
\begin{array}{ll}
h^2 |\ln h| & \text{if $p=1$},\\
h^{2} & \text{if $p>1$}.
\end{array}
\right.
\end{align*}
The estimate for the negative part can be proved in a similar fashion. 

Due to the regularity assumption of $u$, a Taylor expansion
shows that $|\delta_e v_h(x_i)|\le C_2$ for all $x_i\in \Nhi$, where $C_2>0$ depends
on $\|u\|_{C^{1,1}(\bar\Omega)}$.  
Moreover, from the $L_\infty$ error estimate, Proposition \ref{prop:LinftyEstimate} 
and the assumption \eqref{nodalassumption} that interior nodes are at least $h/2$ away from the boundary,
we deduce that 
\[
  \delta_e^+(v_h - u_h)(x_i) \leq C_{\infty}\qquad \forall x_i\in \Nhi,
\]
where the constant $C_\infty>0$ depends on $\|u\|_{C^{3,1}(\bar\Omega)}$.

Let $\tau_k = C_2 k^{1/p} h^2$, and define the set
\[
  A_k := \{ x_i \in \Nhi,  \quad  \delta^+_e (v_h - u_h)(x_i) \geq  \tau_k \}.
\]
%
%
By Lemma \ref{lem:L1norm}
with $s_h(x_i) = |\delta_e^+(v_h-u_h)(x_i)|^p$, $\sigma = C^p_2 h^{2p}$,
 and $M= C_\infty^p$, we obtain
\begin{align}\label{proof:W21sum}
\sum_{x_i\in \Nhi} f_i |\delta_e^+ (v_h-u_h)(x_i)|^p
\le &\; C h^{2p} \left( \mu(\Nhi) + \sum_{k=1}^{ C h^{-2p}} \mu(A_k) \right).
\end{align}

We aim to estimate the measure of set $\mu(A_k)$. 
Due to the relations of the second order difference and contact set given in Remark \ref{rmk:observation},
we have $A_k \subset S_{\epsilon_k} = \Nhi\setminus \mathcal{C}_{\eps_k}$
with $\epsilon_k\in (0,1)$ satisfying $\tau_k = \epsilon_k/(1-\epsilon_k)$.
Therefore, by the estimate \eqref{eqn:measureestimate2} given
in  Proposition \ref{prop:measureestimate},  
\begin{align*}
\mu(A_k) \le 
\mu(S_{\epsilon_k})\le \frac{C_f}{\tau_k} \|g^{1/d}-f^{1/d}\|_{\ell^d(\calC_{\epsilon_k})}
 = \frac{C_f}{k^{1/p} h^2} \|g^{1/d}-f^{1/d}\|_{\ell^d(\calC_{\epsilon_k})}.
\end{align*}
From the concavity of $t\to t^{1/d}$, we have
$(t+\epsilon)^{1/d}-t^{1/d}\le d^{-1} t^{1/d-1} \epsilon$.
Setting $t = g_i$ and $\epsilon = f_i-g_i\ge 0$, we get
\begin{align*}
|f_i^{1/d}-g_i^{1/d}| = f_i^{1/d} - g_i^{1/d} \le d^{-1} g_i^{1/d-1} (f_i - g_i)\le Ch^3
\end{align*}
 due to the consistency error \eqref{eqn:IdealConsistencyError123}
 and the lower bound $g_i \ge Ch^d$.  Consequently, we find that
\[
  \norm{f^{1/d} - g^{1/d}}_{\ell^d(\calC_{\epsilon_k})} \leq C h^2, 
\]
and therefore $\mu(A_k)\le \frac{C}{k^{1/p}}$.
Applying this bound in \eqref{proof:W21sum}, we derive the estimate
\begin{align*}
\sum_{x_i\in \Nhi} f_i |\delta_e^+ (u_h-v_h)(x_i)|^p \leq &\; C h^{2p} \sum_{k = 1}^{C h^{-2p}} \frac 1  {k^{1/p}}\le 
C\left\{
\begin{array}{ll}
h^2 |\ln h| & \text{if $p=1$},\\
h^{2} & \text{if $p>1$}.
\end{array}
\right.
\end{align*}
This completes the proof.
\end{proof}

\begin{remark}
It is worth mentioning that the assumption on the consistency error \eqref{consistencyerror} holds for nodes bounded away from the boundary $\p \dm$ provided that $u \in C^{3,1}(\dm)$. 
However, for nodes close to the boundary $\p \dm$, such an estimate holds only for structured domain, such as a rectangle domain; see the first numerical experiment in Section \ref{S:numerics}. In general, 
this estimate may not be true. 
In fact, Lemma \ref{lem:QuadraticConsistency} 
shows that the (relative) consistency error, $\mathcal{O}(h)$ away from the boundary, is of order $\mathcal{O}(1)$.
In the following subsection, we take into account the lack of consistency
in the boundary layer.
\end{remark}

\subsection{Estimate on general domain}

To this end, we define the barrier nodal function
\[
b_h(x_i) = 
\begin{cases}
 -h^2    \quad &\text{if } x_i \in \Nhi,
\\
 0        \quad &\text{if } x_i \in \Nhb,
\end{cases}
\]
which will be used to ``push down'' 
the graph of the nodal interpolant of $u$
and as such, develop error estimates
in a general setting.
%


\begin{theorem}
\label{thm:W2d}
Let $u \in C^{3,1}(\bar\dm)$ be the solution of the \MA equation \eqref{MA} with $0 < \lambda I \leq D^2 u \leq \Lambda I$,
and assume that   the nodal set $\Nhi$ translation invariant and that $f$ is uniformly positive on $\Omega$. Then the error
in the weighted $W^{2,p}$-norm satisfies
\[
\|N_h u - u_h\|_{W^2_p(\Nhi)} \leq
C\begin{cases}
  h^{1/p}   \quad &\mbox{if $p > d$,}
\\
h^{1/d} \big(\ln\left(\frac 1 h \right)\big)^{1/d}  \quad &\mbox{if $p \le d$,}
\end{cases}
\] 
where $N_h u$ is the interpolation of $u$ on the nodal set $\Nh$ and the constant $C$ 
depends on $\|u\|_{C^{3,1}(\bar\dm)}$, the dimension $d$, and the constant $p$.
\end{theorem}

\begin{proof}
We define the boundary layer:
\[
\dm_h := \{ x_i \in \Nhi, \dist(x_i, \bdry) \leq R h \},
\]
where the constant $R$ is the constant in the consistency error, Lemma \ref{lem:QuadraticConsistency}, which depends on the ellipticity constants $\lambda$ and $\Lambda$ of $D^2 u$.
We set
\[
v_h = N_h u - Cb_h,\qquad g_i  = |\p v_h(x_i)|,
\]
where the constant $C>0$ is sufficiently
large so that $u_h \le v_h$; see
Proposition \ref{prop:LinftyEstimate}.
It is clear from the definition of $b_h$
that 
\[
  |\p v_h(x_i)| = |\p N_h u(x_i)| \quad \mbox{for any $x_i \in \Nhi \setminus \Omega_h$}
\]
and 
\[
  |\p N_h u(x_i)| \geq |\p v_h(x_i)| \geq 0  \quad \mbox{for any $x_i \in \Omega_h$}.
\]
This implies that 
$|f_i-g_i|\le C h^{2+d}$ in $\Nhi\setminus \Omega_h$
and $|f_i-g_i|\le C h^d$ in $\Omega_h$.
We have that $|\delta_e v_h(x_i)|\le C_2$
and $|\delta_e(v_h-u_h)(x_i)|\le C_\infty$
for all $x_i\in \Nhi$.
As in Theorem \ref{thm:C4case}, we shall prove the estimate for the positive part: 
\begin{align*}
\sum_{x_i \in \Nhi} f_i  \left( \delta_e^+ (v_h - u_h)(x_i) \right)^p 
\leq 
\begin{cases}
  C h   \quad &\mbox{if $p > d$,}
\\
C h \ln\left(\frac 1 h \right)  \quad &\mbox{if $p = d$.}
\end{cases}
\end{align*}
The estimate for the negative part can be proved in a similar fashion. 
Also note that the estimate for $p < d$ follows from the estimate of $p = d$ and  H\"older's inequality:
\[
  \|N_h u - u_h\|_{W^2_p(\Nhi)} \leq C_{\mu} \|N_h u - u_h\|_{W^2_d(\Nhi)} \quad \mbox{where } C_\mu := \mu(\Nhi)^{1/p - 1/d} .
\]

We set  $\tau_k = C_2 k^{1/p} h$
and define the set
\[
    A_k := \{ x_i \in \Nhi,  \quad  \left(\delta^+_e (v_h - u_h)(x_i)\right) \geq \tau_k \}. 
\]
Then, by similar arguments as in Theorem \ref{thm:C4case}, 
we find by Lemma \ref{lem:L1norm} that
\begin{align}
\label{proof:W2d}
\sum_{x_i\in \Nhi} f_i \left( \delta_e^+ (v_h-u_h)(x_i) \right)^p 
\le &\; C_2 h^p \left( \mu(\Nhi) + \sum_{k=1}^{ h^{-p}} \mu(A_k) \right).
\end{align}

To estimate the measure of set $\mu(A_k)$, we note that $A_k \subset S_{\epsilon_k} = \Nhi \setminus \mathcal{C}_{\epsilon_k}$
with $\tau_k = \epsilon_k/(1 - \epsilon_k)$.
Invoking the estimate of the measure of the non-contact set $S_{\epsilon}$
stated in Proposition \ref{prop:measureestimate}, we obtain 
\[
\mu(A_k) \leq \mu (S_{\epsilon_k}) \leq \nu_{\tau_k}(\mathcal{C}_{\epsilon_k}) - \mu(\mathcal{C}_{\epsilon_k}).
\]
We then
divide the estimate of $\nu_{\tau_k}(\mathcal{C}_{\epsilon_k}) - \mu(\mathcal{C}_{\epsilon_k}) $ into two parts:
\begin{align*}
\nu_{\tau_k}(\mathcal{C}_{\epsilon_k}) - \mu(\mathcal{C}_{\epsilon_k}) = 
&\;
\sum_{x_i \in \mathcal{C}_{\epsilon_k}}\left[
\left( f_i^{1/d} + \frac 1 {\tau_k} e_i^{1/d} \right)^d - f_i\right]
\\
= &\;
\left( 
\sum_{x_i \in \mathcal{C}_{\epsilon_k} \cap \dm_{h}} + \sum_{x_i \in \mathcal{C}_{\epsilon_k} \setminus \dm_{h}} 
\right)
\left[\left( f_i^{1/d} + \frac 1 {\tau_k} e_i^{1/d} \right)^d - f_i\right],
\end{align*}
where we recall that $e_i^{1/d} = f_i^{1/d}-g_i^{1/d}$.
Since $f_i^{1/d} = \mathcal{O}(h)$ and $g_i^{1/d}= \mathcal{O}(h)$, we have
\begin{align*}
\left|\left( f_i^{1/d} + \frac 1 {\tau_k} e_i^{1/d} \right)^d - f_i\right|
&\le \frac{d}{\tau_k} \max\{\big| f_i^{1/d} + \frac{1}{\tau_k} e_i^{1/d}\big|,f_i^{1/d}\}^{d-1} |e_i^{1/d}|\\
&\le  \frac{C h^{d-1}}{\tau_k^d} |e_i^{1/d}|.
\end{align*} 

In the set $\mathcal{C}_{\epsilon_k} \cap \dm_{h}$,  the consistency error satisfies $|e_i^{1/d}| = \mathcal{O}(h)$; 
see Lemma \ref{lem:QuadraticConsistency}. Therefore, we have
\[
\left| 
\left( f_i^{1/d} + \frac 1 {\tau_k} e_i^{1/d} \right)^d - f_i
\right| 
\leq \frac{C h^d} {\tau_k^d}\qquad \forall x_i \in \mathcal{C}_{\epsilon_k}\cap \Omega_h.
\]
On the other hand,
in the set $\mathcal{C}_{\epsilon_k} \setminus \dm_{h}$, we conclude
as in Theorem \ref{thm:C4case}, that $|e_{i}^{1/d}| = \mathcal{O}(h^3)$, and
\[
\left| 
\left( f_i^{1/d} + \frac 1 {\tau_k} e_i^{1/d} \right)^d - f_i
\right| 
\leq \frac{C h^{2+d}} {\tau_k^d}.
\]

Combining both estimate and applying the fact that $\# (\mathcal{C}_{\epsilon_k}\cap \Omega_h) \leq C h^{1-d}$
and $\# (\mathcal{C}_{\epsilon_k} \setminus \Omega_h)\le Ch^{-d}$, we obtain
\[
\nu_{\tau_k}(\Ceb) - \mu(\Ceb) 
\leq
\frac{C h}{\tau_k^d} + \frac{C h^2}{\tau_k^d}
\leq 
\frac{C h}{\tau_k^d} 
\]
because $h \leq 1$. 
Hence, we conclude that 
\[
\mu(A_k) \leq \frac{C h}{\tau_k^d} .
\]

Applying this estimate to  \eqref{proof:W2d}, we  arrive at
\[
  \sum_{x_i \in \Nhi} f_i |\delta_e^+ (v_h - u_h)(x_i)|^p \leq C_2 h^p \sum_{k = 1}^{h^{-p}} \frac{h}{h^d k^{d/p}}.
\]
Since 
\[
\sum_{k = 1}^{h^{-p}} \frac{1}{k^{d/p}} \leq 
\begin{cases}
  C(d, p) h^{d - p}    \quad &\mbox{if $p>  d$,}
\\
C \ln \left(\frac 1 h \right) \quad &\mbox{if $p = d$,}
\end{cases}
\]
we conclude that 
\[
\sum_{x_i \in \Nhi} f_i |\delta_e^+ (v_h - u_h)(x_i)|^p 
\leq 
\begin{cases}
  C h    \quad &\mbox{if $p>d$,}
\\
C h \ln\left(\frac 1 h \right) \quad &\mbox{if $p = d$.}
\end{cases}
\]
Finally we note that by H\"older's inequality, there holds for $p<d$,
\begin{align*}
\|v_h\|_{W^2_p(\Nhi)}  
&= \Big(\sum_{x_i\in \Nhi} f_i |\delta_e v_h(x_i)|^p\Big)^{1/p}\\
&\le \big(\sum_{x_i\in \Nhi} f_i |\delta_e v_h(x_i)|^d\Big)^{1/d}\Big( \sum_{x_i\in \Nhi} f_i \Big)^{(d-p)/(dp)}\le C\|v_h\|_{W^2_d(\Nhi)}.
\end{align*}
This completes the proof.
\end{proof}

%
           \section{Numerical experiments}\label{S:numerics}
%
In this section, we perform  numerical examples to illustrate  the accuracy of the method,
and to compare the results with the theory.
In the tests, we replace the homogeneous boundary condition \eqref{MA2}
with $u=g$ on $\partial \Omega$.  The theoretical results
developed in the previous sections can be applied to this 
slightly more general problem with minor modifications.

We consider three different test problems, each reflecting different scenarios of regularity. 
Each set of problems is performed in two dimensions ($d=2$), and errors are reported
in the (discrete) $L_\infty$, $H^1$, $W_1^2$, and $W^2_2$ norms.   Here, 
a nine-point stencil is used in the definition of the $W^2_p$ norms 
with $e_1 = (1,0)$, $e_2 = (0,1)$, $e_3 =(1,1)$ and $e_4 = (1, -1)$.  That is, 
with an abuse of notation, we set
\[
  \|v\|_{W^2_p(\Nhi)}^p = \sum_{j=1}^{4} \sum_{x_i \in \Nhi} |\delta_{e_j}v(x_i)|^p.
\]
As explained in \cite{NochettoZhang17} and in Section \ref{sec:Method}, 
a convex nodal function induces a triangulation of $\Omega$
whose set of vertices corresponds to $\Nh$.
For a computed solution $u_h$, we associate with it
a piecewise linear polynomial on the induced mesh, which we still
denote by $u_h$, and use the quantity $\|u-u_h\|_{H^1(\Omega)}$
to denote the $H^1$ error in the experiments below.

A summary of the theoretical results in Sections \ref{sub:BM} and \ref{S:W21estimate} when $d=2$ are  
\begin{align*}
\|N_h u-u_h\|_{L_\infty(\Nhi)} = \mathcal{O}(h^2),\qquad
\|N_h u-u_h\|_{W^2_p(\Nhi)} = \mathcal{O}(h^{1/2-\epsilon}) ,\ p=1,2
\end{align*}
for any $\epsilon>0$,
  provided that $u \in C^{3,1}(\bar\dm)$.

\subsection*{Example I: Smooth Solution $u\in C^\infty(\bar\Omega)$}
We consider the example
\begin{align}
\label{eqn:SmoothSolnExample}
u(x,y) = e^{\frac{x^2 + y^2}2},\quad f(x,y) = (1 + x^2 + y^2) e^{x^2 + y^2},\quad \text{and }\Omega = (-1,1)^2,
\end{align}
and list the resulting errors and rates of the scheme in Table \ref{Table:Example1}.
The Table clearly shows that the errors decay with rate $\mathcal{O}(h^2)$
in all norms.  This behavior matches the theoretical results
of Proposition \ref{prop:LinftyEstimate}, but indicates
that the $W^2_p$ estimates stated in Theorem \ref{thm:W2d} are not sharp.

\begin{table}[h]
\begin{center}
\begin{tabular}[t]{ | l | c | c | c | c | c | c | c | c |}
\hline
     $h$               & $L_{\infty}$ & rate  & $H^1$ & rate & $W^2_1$ & rate & $W^2_2$ & rate \\
 \hline
1 & 1.12e-01 & 0.00 & 2.24e-01 &  & 4.49e-01 &  & 1.44e+01 &  
\\ 
 \hline 
 1/2 & 4.78e-02 & 1.23 & 1.35e-01 & 0.73 & 6.02e-01 & -0.42 & 4.24e-01 & 5.08 
\\ 
 \hline 
 1/4 & 1.37e-02 & 1.80 & 4.35e-02 & 1.63 & 2.94e-01 & 1.03 & 1.93e-01 & 1.13 
\\ 
 \hline 
 1/8 & 3.55e-03 & 1.95 & 1.16e-02 & 1.91 & 9.93e-02 & 1.57 & 6.34e-02 & 1.61 
\\ 
 \hline 
 1/16 & 8.96e-04 & 1.99 & 2.94e-03 & 1.98 & 2.86e-02 & 1.80 & 1.80e-02 & 1.82 
\\ 
 \hline 
 1/32 & 2.24e-04 & 2.00 & 7.39e-04 & 1.99 & 7.66e-03 & 1.90 & 4.79e-03 & 1.91 
\\ 
 \hline 
 1/64 & 5.61e-05 & 2.00 & 1.85e-04 & 2.00 & 1.98e-03 & 1.95 & 1.24e-03 & 1.95 
\\ 
 \hline 
\end{tabular}
\label{Table:Example1}
\caption{Rate of convergence for a  smooth solution (Example I).}
\end{center}
\end{table}

\subsection*{Example II: Piecewise Smooth Solution $u \in W^{2}_{\infty}$}
In this example, the domain is 
$\Omega= (-1,1)^2$, and the exact solution
and data are taken to be
\begin{align*}
 u(x)  =& \left\{
\begin{array}{ll}
2 |x|^2  \quad &\text{in $|x| \leq 1/2$},  
\\
 2(|x| - 1/2)^2 + 2 |x|^2 \quad &\text{in $ 1/2 \leq |x| $},
\end{array} \right.
\\
 f(x)  =& \left\{
\begin{array}{ll}
16  \quad \quad &\text{in $|x| \leq 1/2$},  
\\
64 -  16 |x|^{-1} \quad \qquad\hspace{0.59cm}&\text{in $ 1/2 \leq |x| $}.
\end{array} \right.
\end{align*}

A simple calculation shows that $u \in C^{1,1}(\bar\Omega)$ and $u \in C^4(\dm \setminus \partial B_1)$,
but $u\not\in C^2(\bar\Omega)$.
The errors and rates of convergence are given in Table \ref{Table:Example2}.
The table shows that, while all errors tend to zero as the mesh is refined, 
the rates of convergence in the $L_\infty$ and $W^2_1$ norms are less obvious than the previous set of experiments.
Nonetheless, while Theorem \ref{thm:W2d} assumes more regularity 
of the exact solution, we do observe a convergence rate of approximately $\mathcal{O}(h^{1/2})$
in the $W^2_2$ as stated in the theorem.

\begin{table}[h]
\begin{center}
\begin{tabular}[t]{ | l | c | c | c | c | c | c | c | c |}
\hline
     $h$               & $L_{\infty}$ & rate  & $H^1$ & rate & $W^2_1$ & rate & $W^2_2$ & rate \\
\hline 
1 & 4.02e-01 & 0.00 & 8.04e-01 & 0.00 & 1.61 & 0.00 & 1.61 & 0.00 
\\ 
 \hline 
 1/2 & 4.19e-02 & 3.26 & 1.30e-01 & 2.63 & 6.08e-01 & 1.40 & 5.39e-01 & 1.58 
\\ 
 \hline 
 1/4 & 2.89e-02 & 0.53 & 6.84e-02 & 0.92 & 6.46e-01 & -0.09 & 5.54e-01 & -0.04 
\\ 
 \hline 
 1/8 & 1.27e-02 & 1.18 & 3.50e-02 & 0.97 & 5.14e-01 & 0.33 & 4.54e-01 & 0.29 
\\ 
 \hline 
 1/16 & 4.58e-03 & 1.47 & 1.38e-02 & 1.34 & 2.76e-01 & 0.90 & 3.15e-01 & 0.53 
\\ 
 \hline 
 1/32 & 8.02e-04 & 2.51 & 3.59e-03 & 1.94 & 1.08e-01 & 1.35 & 2.08e-01 & 0.60 
\\ 
 \hline 
 1/64 & 4.33e-04 & 0.89 & 1.50e-03 & 1.26 & 6.36e-02 & 0.77 & 1.56e-01 & 0.42
\\ 
 \hline 
\end{tabular}
\label{Table:Example2}
\caption{Rate of convergence of piecewise smooth viscosity solution (Example II).}
\end{center}
\end{table}

\subsection*{Example III: Singular Solution $u\in W^2_p$ with $p < 2$}
In the last series of experiments, the domain is $\Omega = (-1,1)^2$,
and the solution and data are
\begin{align*}
 u(x) & = \left\{
\begin{array}{ll}
x^4 + \frac 32 y^2/x^2  \quad &\text{in $|y| \leq |x|^3$},  
\\
 \frac 12 x^2 y^{2/3} + 2 y^{4/3} \quad &\text{in $ |y| \geq |x|^3 $},
\end{array} \right.\\
\\
 f(x)  =& \left\{
\begin{array}{ll}
36-9y^2/x^6 \quad\ \hspace{1cm} &\text{in $|y|\leq |x|^3$},\\
\frac89 - \frac59 x^2/y^{2/3} \quad &\text{in $|y|>|x|^3$}.  
\end{array}
\right.
\end{align*}
This example is constructed in \cite{Wang95} to show that $D^2 u(x)$ may not be in $W^2_p$ for large $p$ for discontinuous $f$. 
The errors
of the method for this problem
are listed in Table \ref{Table:Example3}.
Because the exact solution does not
enjoy $W^2_2$ regularity, 
it is not expected that the discrete solution will converge
in the discrete $W^2_2$ norm, and this is observed
in the table.
However, we do observe 
convergence in the $L_\infty$, $H^1$, and $W^2_1$ norms
with approximate rates
$\|N_h u- u_h\|_{L_\infty(\Nhi)} = \mathcal{O}(h^{4/3})$,
$\|N_h u-u_h\|_{H^1(\Nhi)} = \mathcal{O}(h)$,
and $\|N_h u-u_h\|_{W^2_1(\Nhi)} = \mathcal{O}(h^{1/2})$.

\begin{table}[h]
\begin{center}
\begin{tabular}[t]{ | l | c | c | c | c | c | c | c | c |}
\hline
     $h$               & $L_{\infty}$ & rate  & $H^1$ & rate & $W^2_1$ & rate & $W^2_2$ & rate \\
\hline 
1 & 8.36e-01 & 0.00 & 1.67 & 0.00 & 3.35 & 0.00 & 3.35 & 0.00 
\\ 
 \hline 
 1/2 & 2.34e-01 & 1.84 & 9.11e-01 & 0.88 & 5.48 & -0.71 & 3.94 & -0.24 
\\ 
 \hline 
 1/4 & 1.86e-01 & 0.33 & 4.80e-01 & 0.92 & 4.90 & 0.16 & 4.02 & -0.03 
\\ 
 \hline 
 1/8 & 8.52e-02 & 1.13 & 2.41e-01 & 1.00 & 4.00 & 0.29 & 3.94 & 0.03 
\\ 
 \hline 
 1/16 & 3.41e-02 & 1.32 & 1.02e-01 & 1.24 & 2.38 & 0.75 & 3.33 & 0.24 
\\ 
 \hline 
 1/32 & 1.35e-02 & 1.34 & 4.79e-02 & 1.09 & 1.59 & 0.58 & 3.17 & 0.07 
\\ 
 \hline 
\end{tabular}
\label{Table:Example3}
\caption{Rate of convergence of $W^2_p$ solution with $p<2$ (Example III).}
\end{center}
\end{table}

%
%
%


%
\end{document}